\documentclass[11pt]{amsart}
\usepackage{amsmath,amsfonts,amssymb,amscd,amsthm,amsbsy,upref,graphicx,mathtools,longtable,mdframed,bm,xcolor,cite,setspace,soul,color}
\usepackage[bookmarks]{hyperref}
\usepackage{tikz}
\usetikzlibrary{cd}
\usepackage[lmargin=1in,rmargin=1in,tmargin=1in,bmargin=1in]{geometry}

\newtheorem{thm}{Theorem}[section]
\newtheorem{cor}[thm]{Corollary}
\newtheorem{lemma}[thm]{Lemma}

\theoremstyle{definition}

\theoremstyle{remark}
\newtheorem{remark}[thm]{Remark}

\numberwithin{equation}{section}
\numberwithin{thm}{section}

\newcommand{\laplacian}{\Delta}
\newcommand{\grad}{\nabla} 
\renewcommand{\eqref}[1]{\textup{{\normalfont(\ref{#1}}\normalfont)}}

\newcommand{\norm}[1]{\left\lVert#1\right\rVert}
\def\be{\begin{equation}}
\def\ee{\end{equation}}

\newcommand{\abs}[1]{\left\vert#1\right\vert}
\newcommand{\R}{{\mathbb{R}}}

\newcommand{\lap}{\Delta}

\newcommand{\curl}{\mbox{curl}\,}

\newcommand{\defneq}{=}



\date{}

\begin{document}

\title[Dichotomy for emNLS]{Dichotomy Results for the Electromagnetic\\ Schr\"odinger Equation}

\author[M. Czubak]{Magdalena Czubak}
\address{Department of Mathematics, University of Colorado Boulder}
\email{czubak@math.colorado.edu}

\author[I. Miller]{Ian Miller}
\address{Department of Mathematics, University of Colorado Boulder}
\email{ian.d.miller@colorado.edu}

\author[S. Roudenko]{Svetlana Roudenko}
\address{Department of Mathematics and Statistics, Florida International University}
\email{sroudenko@fiu.edu}

\maketitle

\begin{abstract}
The electromagnetic nonlinear Schr{\"{o}}dinger (emNLS) equation is a variant of the well-studied nonlinear Schr{\"{o}}dinger equation. In this article, we consider questions of global existence or blow-up for emNLS in dimensions 3 and higher.
\end{abstract}

\section{Introduction}
We consider the electromagnetic nonlinear Schr{\"{o}}dinger equation (emNLS) and investigate global existence versus blow-up in finite time. Mathematically, emNLS is a natural generalization of the NLS equation when one replaces regular derivatives by covariant derivatives. In physics, the equation has connections to the quantum electrodynamics and Bose-Einstein condensates, e.g., see \cite{Erdos,Lieb2006}.

In this article, we are interested in the focusing equation and making progress in obtaining a general understanding of the global theory similar to what is known for the standard NLS equation. Here, ``standard" refers to the case when all the potentials are zero.

In that case, 
\begin{equation}\label{E:NLS}
i\partial_t u + \Delta u + |u|^{p-1} u = 0, \quad (x, t) \in \mathbb R^n \times \mathbb R,
\end{equation}
where $u = u(x, t)$ is complex-valued and $p > 1$. The solutions of this equation satisfy several conservation laws. In particular, the $L^2$ norm, which is mass/charge, is conserved, and we also have conservation of energy 
\begin{equation}\label{E:E}
E[u]=\frac 12\norm{\nabla u}^2_{L^2(\mathbb{R}^n)}-\frac 1{p+1}\norm{u}_{L^{p+1}(\mathbb{R}^n)}^{p+1}.
\end{equation}

The NLS equation has scaling: $u_\lambda(x,t) = \lambda^{\frac2{p-1}} u(\lambda x, \lambda^2 t)$ is a solution if $u(x, t)$ is as
well. This scaling produces a scale-invariant Sobolev norm $\dot{H}^{s_c}(\mathbb{R}^n)$ with
\begin{equation}\label{E:sc}
s_c=\frac n2 - \frac{2}{p-1}.
\end{equation}

The NLS equation admits a soliton solution, which we call $u_Q$ (see Section \ref{GNQ}). This solution acts as a threshold for when solutions exhibit a dichotomy between scattering and blow-up in both time directions. For a precise statement of this fact, we consider the scale invariant quantities
\begin{align*}
& E[u]^{s_c} \norm{u}_{L^2(\mathbb{R}^n)}^{2(1-s_c)},\\
& \norm{\grad u}_{L^2(\mathbb{R}^n)}^{s_c} \norm{u}_{L^2(\mathbb{R}^n)}^{(1-s_c)}.
\end{align*}
The set of initial data satisfying $E[u]^{s_c} \norm{u_0}_{L^2(\mathbb{R}^n)}^{2(1-s_c)} < E[u_Q]^{s_c} \norm{u_Q}_{L^2(\mathbb{R}^n)}^{2(1-s_c)}$ has the following dichotomy.
\begin{thm}\cite{HR07,HR08,DHR08}\label{NLSd}
Let $u_0\in H^1(\R^n)$ and let the critical exponent $s_c$ satisfy $0<s_c<1$. Suppose 
\be\label{eless}
E[u_0]^{s_c}\norm{u_0}^{2(1-s_c)}_{L^2(\R^n)}<E[u_Q]^{s_c}\norm{u_Q}^{2(1-s_c)}_{L^2(\R^n)}.
\ee
\begin{itemize}
\item If 
\be\label{nablaless}
\norm{\nabla u_0}^{s_c}_{L^2(\R^n)}\norm{u_0}^{1-s_c}_{L^2(\R^n)}<\norm{\nabla u_Q}^{s_c}_{L^2(\R^n)}\norm{u_Q}^{1-s_c}_{L^2(\R^n)},
\ee
then the solution $u$ is global in both time directions and
\be\label{nablalesst}
\norm{\nabla u(t)}^{s_c}_{L^2(\R^n)}\norm{u_0}^{1-s_c}_{L^2(\R^n)}<\norm{\nabla u_Q}^{s_c}_{L^2(\R^n)}\norm{u_Q}^{1-s_c}_{L^2(\R^n)},
\ee
for all time $t \in \R$. Moreover, the solution scatters in $H^1(\mathbb R^n)$.
\item If 
\be\label{nablamore}
\norm{\nabla u_0}^{s_c}_{L^2(\R^n)}\norm{u_0}^{1-s_c}_{L^2(\R^n)}>\norm{\nabla u_Q}^{s_c}_{L^2(\R^n)}\norm{u_Q}^{1-s_c}_{L^2(\R^n)},
\ee
then 
\be\label{nablamoret}
\norm{\nabla u(t)}^{s_c}_{L^2(\R^n)}\norm{u_0}^{1-s_c}_{L^2(\R^n)}>\norm{\nabla u_Q}^{s_c}_{L^2(\R^n)}\norm{u_Q}^{1-s_c}_{L^2(\R^n)},
\ee
and if $\abs{x}u_0\in L^2$ or $u$ is radial, then the solution blows up in finite time in both time directions.
\end{itemize}
\end{thm}
\begin{remark}
For further results regarding the blow up scenario, see e.g. \cite{HR10}.
\end{remark}
In this article, we wish to establish a dichotomy result regarding global existence and blow-up for the electromagnetic equation.

There is extensive literature on the existence of smoothing and Strichartz estimates for emNLS (see for example \cite{GST07,EGS09, JFA}). Several authors have also studied wellposedness and scattering \cite{CazenaveEsteban88, DeBouard91, Leger}. Much less is known, however, about blow-up and formation of singularities.
To our knowledge, the only articles considering blow-up are the works of Goncalves Ribero \cite{goncalves}, Garcia \cite{garcia}, Kieffer and Loss \cite{kieffer&loss}, and Dinh \cite{dinh}. The first two articles consider negative energy blow-up using the virial identity argument. The work of Kieffer and Loss \cite{kieffer&loss} develops a modification of the virial identity argument, with a new conserved quantity related to the angular momentum. All articles consider constant magnetic field with the exception of Garcia, who besides considering constant magnetic field, considers potentials with a non-constant magnetic field, but having a zero trapping component (see below). The articles \cite{goncalves, garcia, kieffer&loss} focus on blow-up results, while \cite{dinh} considers both global wellposedness and blow-up and obtains dichotomy results in the spirit of \cite{HR07}. Among these articles, most have focused on the case with no electric potential. However, Garcia \cite{garcia} considers a class of both magnetic and electric potentials.

The goal of the present article is to extend the results of \cite{goncalves, garcia, kieffer&loss, dinh} to potentials with non-constant magnetic field and with nonzero trapping component. Moreover, the results address both global wellposedness and blow-up in finite time for positive or negative energy, and are stated as a dichotomy result.

We work in the Coulomb gauge. The proof of global existence can go through similar to the standard NLS equation. The main obstacle comes while establishing the blow-up. This is due to having extra terms in the virial identity (see \eqref{virial}). In the case of the purely magnetic equation, this can be handled by a careful analysis of the remaining error term.

We also include results for the full electromagnetic equation.
To obtain the dichotomy statement, the result requires smallness of the potentials, and a smaller bound on the mass-energy, i.e., the bound \eqref{eless2}. If we are only interested in a global wellposedness result, then no smallness is required, and we can get as close as we want to an electromagnetic analog of the ground state quantity. We state this as a corollary.

\subsection{Main Results}
Let $n\geq 3$. Consider the focusing magnetic nonlinear Schr\"odinger equation, given by
\begin{equation}\label{emNLS}
    \begin{cases}
        iD_{t}u+ \lap_{A} u =-\abs{u}^{p-1}u, \quad x\in\mathbb R^n, ~t\in\R,\\
        u(0,x)=u_0(x),
    \end{cases}
\end{equation}
where 
\begin{align*}
&u: \R^{n+1}\mapsto \mathbb C,\\
&A_{\alpha}: \R^{n}\mapsto \R,\quad \alpha=0,\cdots, n,\\
&D_\alpha =\partial_\alpha + iA_\alpha,\quad \alpha=0,\cdots, n, \quad D_{t}=D_{0},\\
&\lap_{A}=D^{2}=D_{j}D_{j}=(\partial_j + iA_j)(\partial_j + iA_j)=\lap+2iA\cdot\nabla+i\nabla\cdot A-\abs{A}^{2}.
\end{align*}
The Greek indices range from $0$ to $n$, and Roman indices range from $1$ to $n$. We sum over repeated indices.
We use $\nabla_A u$ to denote
\[
(D_1 u, \dots, D_n u)
\]
and sometimes write $(A_0, A)$ to mean $(A_0, A_1,\dots, A_n)$. The power $p$ satisfies
 \be\label{p}
 1+\frac 4n< p < 1+\frac 4{n-2}.
 \ee
The equation emNLS is scaling invariant (see Section \ref{Pdef} below) and \eqref{p} is equivalent to the critical Sobolev exponent $s_c$ satisfying the intercritical range
$$
0< s_c<1.
$$
We note that this correspondence of the range of $s_c$ and $p$ is the same as in the case of the standard NLS equation.
\smallskip

We next define the spaces adapted to our magnetic NLS setting.
Let
\[
H^1_A(\R^n)=\{ f\in L^2(\R^n): \nabla_A f\in L^2(\R^n) \},
\]
 and
 \[
 \Sigma_A(\R^n)=\{ f\in H^1_A(\R^n): \abs{x} f\in L^2(\R^n) \}.
 \]
The energy (or Hamiltonian) for emNLS is given by 
 \[
E_A[u]=\frac 12\norm{\nabla_Au}^2_2-\frac 1{p+1}\norm{u}_{p+1}^{p+1}+\frac 12\int A_0 \abs{u}^2dx,
\]
and is conserved when the equation has a satisfactory local theory for data in $H^1_A$.

In the case of 3 dimensions, for a magnetic potential $A$, the magnetic field can be identified with $\curl A$. It was observed in \cite{FanelliVega09} that 
\[
B_{\tau}=\frac{x}{\abs{x}}\times \curl A
\]
is an obstruction to dispersion. The term $B_{\tau}$ is called the \emph{trapping component} of the magnetic field.

In higher dimensions, the trapping component can be generalized to
\[
B^{T}_{\tau}=\frac{x^{T}}{\abs{x}} (F_{jk}),
\]
where $(F_{jk})$ is a matrix with the $(j,k)$ entry given by $F_{jk}=\partial_j A_k-\partial_k A_j$, so the $k$'th entry of $B_{\tau}$ is 
\be\label{btau}
(B_\tau)_k=\frac{x_{j}}{\abs{x}} F_{jk}.
\ee

Our main result is the following theorem.
\begin{thm}\label{main}
Let $n\geq 3$ and $0<s_c<1$. There exists a constant $C_{\tau}$ so that if $A \in L^2_{\text{loc}}(\mathbb{R}^n;\mathbb{R}^n)$ is a potential inducing satisfactory local theory, then we have the following dichotomy result:
Suppose  $u_0 \in H^1_A(\mathbb{R}^n)$ is such that
\be\label{eless}
E_A[u_0]^{s_c}\norm{u_0}^{2(1-s_c)}_{L^2(\R^n)}<E[u_Q]^{s_c}\norm{u_Q}^{2(1-s_c)}_{L^2(\R^n)}.
\ee
\begin{itemize}
\item If 
\be\label{nablaless}
\norm{\nabla_A u_0}^{s_c}_{L^2(\R^n)}\norm{u_0}^{1-s_c}_{L^2(\R^n)}<\norm{\nabla u_Q}^{s_c}_{L^2(\R^n)}\norm{u_Q}^{1-s_c}_{L^2(\R^n)},
\ee
then the solution $u$ is global and 
\be\label{nablalesst}
\norm{\nabla_A u(t)}^{s_c}_{L^2(\R^n)}\norm{u_0}^{1-s_c}_{L^2(\R^n)}<\norm{\nabla u_Q}^{s_c}_{L^2(\R^n)}\norm{u_Q}^{1-s_c}_{L^2(\R^n)},
\ee
for all time $t \in \R$.
\item If 
\be\label{nablamore}
\norm{\nabla_A u_0}^{s_c}_{L^2(\R^n)}\norm{u_0}^{1-s_c}_{L^2(\R^n)}>\norm{\nabla u_Q}^{s_c}_{L^2(\R^n)}\norm{u_Q}^{1-s_c}_{L^2(\R^n)},
\ee
then 
\be\label{nablamoret}
\norm{\nabla_A u(t)}^{s_c}_{L^2(\R^n)}\norm{u_0}^{1-s_c}_{L^2(\R^n)}>\norm{\nabla u_Q}^{s_c}_{L^2(\R^n)}\norm{u_Q}^{1-s_c}_{L^2(\R^n)},
\ee
on the full interval of existence, and if
\begin{equation}\label{E:trap1}
\norm{|x|^2 B_{\tau}}_{L^\infty(\R^n)} < C_{\tau},
\end{equation}
and $u_0$ has finite variance, the solution blows up in finite time.
\end{itemize}
\end{thm}
\begin{remark}
Observe that for potentials satisfying \eqref{E:trap1} this theorem is the magnetic analog of the global well-posedness and blow-up portion of Theorem \ref{NLSd}.
\end{remark}
\begin{remark}
Under the assumption \eqref{eless}, the expressions \eqref{nablaless} and \eqref{nablamore} are the only possibilities. This can be seen, for example, in the proof of Lemma \ref{lem:double-cut}.
\end{remark}
\begin{remark}
    We note that the theorem applies to both positive and negative energy. For negative energy, the applicable scenario is only the blow-up scenario.
\end{remark}
\begin{remark}
We can show that the theorem holds for the value
\[
C_\tau=\min \left( \frac{(n-2)(n(p-1)-4)}{16\, C_{p,n}\norm{u_Q}_2^{\frac{(1-s_c)}{s_c}} \norm{\grad u_Q}_2}\,,
\,\frac{(n-2)(n(p-1)-4)}{16}\right),
\]
where the constant $C_{p,n}$ is obtained in Lemma \ref{lem:GSP}.
This value is not sharp.
\end{remark}

\begin{remark}
The Coulomb gauge and the norm in \eqref{E:trap1} is inspired by assumptions needed to obtain Strichartz estimates. In general, smallness assumptions on $B_\tau$ are needed to obtain Strichartz estimates and show local wellposedness. For example, in \cite{JFA} to obtain Strichartz estimates in dimensions greater than or equal to 4, a smallness condition on $\norm{|x|^2 B_{\tau}}_{\infty}$ is needed (possibly with a different constant from \eqref{E:trap1}), while in 3 dimensions, this appears as a smallness condition on $\norm{|x|^{3/2} B_{\tau} }_{L^2_r L^{\infty} (S_r)}$, where
\be\label{Srnorm}
 \norm{f}^{2}_{L^{2}_{r}L^{\infty}(S_{r})}=\int^{\infty}_{0}\sup_{\abs{x}=r}\abs{f}^{2}dr, 
 \ee
see further details on this in Appendix \ref{app1}. 
\end{remark}

If we include the electric potential, we obtain the following results.
\begin{thm}\label{main2}
Let $n\geq 3$, $0<s_c<1$, and $u_0 \in H^1_A(\mathbb{R}^n)$ such that
\be\label{eless2}
E_A[u_0]\norm{u_0}^{2\frac{(1-s_c)}{s_c}}_{L^2(\R^n)}<(1-a_0)^{1+\frac 2{s_c(p-1)}}E[u_Q] \norm{u_Q}_{L^2(\R^n)}^{2\frac{(1-s_c)}{s_c}}-C(A_0, B_\tau),
\ee
where
\be
C(A_0, B_\tau)=\frac{2a_0+a_1+2b_0}{n(p-1)}x_\ast\quad\mbox{and}\quad x_\ast=(1-a_0)^\frac{2}{s_c(p-1)}\norm{\grad u_Q}_2^2 \norm{u_Q}_2^{2\frac{(1-s_c)}{s_c}},
\ee
and $a_0, a_1, b_0$ are constants appearing in 
\eqref{a0}, \eqref{a1}, and \eqref{b0}, respectively.

If $A, A_0$ are potentials inducing satisfactory local theory, then we have the following dichotomy result:
\begin{itemize}
\item If 
\be\label{nablaless2}
\norm{\nabla_A u_0}^{s_c}_{L^2(\R^n)}\norm{u_0}^{1-s_c}_{L^2(\R^n)}<(1-a_0)^\frac{1}{p-1}\norm{\nabla u_Q}^{s_c}_{L^2(\R^n)}\norm{u_Q}^{1-s_c}_{L^2(\R^n)},
\ee
then the solution $u$ is global and 
\be\label{nablaless2t}
\norm{\nabla_A u(t)}^{s_c}_{L^2(\R^n)}\norm{u_0}^{1-s_c}_{L^2(\R^n)}<(1-a_0)^\frac{1}{p-1}\norm{\nabla u_Q}^{s_c}_{L^2(\R^n)}\norm{u_Q}^{1-s_c}_{L^2(\R^n)},
\ee
for all time $t \in \R$.
\item If
\be\label{nablamore2}
\norm{\nabla_A u_0}^{s_c}_{L^2(\R^n)}\norm{u_0}^{1-s_c}_{L^2(\R^n)}>(1-a_0)^\frac{1}{p-1}\norm{\nabla u_Q}^{s_c}_{L^2(\R^n)}\norm{u_Q}^{1-s_c}_{L^2(\R^n)},
\ee
and if 
\be\label{c1}
n(p-1)a_0+2a_1+4b_0<n(p-1)-4=s_c(p-1),
\ee
then 
\be\label{nablamore2t}
\norm{\nabla_A u(t)}^{s_c}_{L^2(\R^n)}\norm{u_0}^{1-s_c}_{L^2(\R^n)}>(1-a_0)^\frac{1}{p-1}\norm{\nabla u_Q}^{s_c}_{L^2(\R^n)}\norm{u_Q}^{1-s_c}_{L^2( \R^n)},
\ee
on the full interval of existence, and if $u_0$ has finite variance the solution blows up in finite time.
\end{itemize}
\end{thm}

\begin{remark}
We can consider arbitrarily small $C(A_0,B_{\tau})$ by taking potentials with sufficiently small constants $a_0,a_1,b_1$.
\end{remark}
 
The proof of the above theorem implies the following corollary. %next corollary.
\begin{cor}\label{cor1}
    Let $n\geq 3$, $0<s_c<1$, and $A,A_0$ be potentials inducing satisfactory local theory. Suppose $u_0\in H^1_A$ is such that
\be\label{eless4}
E_A[u_0] \norm{u_0}^{2\frac{(1-s_c)}{s_c}}_{L^2(\R^n)}<(1-a_0)^{1+\frac 2{s_c(p-1)}}E[u_Q] \norm{u_Q}_{L^2(\R^n)}^{2\frac{(1-s_c)}{s_c}}.
\ee
If 
\be\label{nablaless4}
\norm{\nabla_A u_0}^{s_c}_{L^2(\R^n)}\norm{u_0}^{1-s_c}_{L^2(\R^n)}<(1-a_0)^\frac{1}{p-1}\norm{\nabla u_Q}^{s_c}_{L^2(\R^n)}\norm{u_Q}^{1-s_c}_{L^2(\R^n)},
\ee
then the solution $u$ is global and 
\be\label{nablaless4t}
\norm{\nabla_A u(t)}^{s_c}_{L^2(\R^n)}\norm{u_0}^{1-s_c}_{L^2(\R^n)}<(1-a_0)^\frac{1}{p-1}\norm{\nabla u_Q}^{s_c}_{L^2(\R^n)}\norm{u_Q}^{1-s_c}_{L^2(\R^n)},
\ee
for all time $t \in \R$.
\end{cor}

The paper is organized as follows. In Section \ref{S:2} we discuss preliminaries regarding the standard NLS as well as emNLS. Then, in Section \ref{S:3} we prove our main lemmas, including a remainder identity for Kato's inequality which provides control of the trapping term, as well as refined bounds for the kinetic energy of solutions, and bounds regarding the proximity of a solution to the ground state. In Section \ref{S:4} we prove Theorem \ref{main} by applying the remainder identity for Kato's inequality to reduce the problem to the control of a nonlinear error term, which we then bound using refined control of kinetic energy. In Section \ref{S:5} we demonstrate a way to modify dichotomy arguments for NLS to obtain results for emNLS, and prove Theorem \ref{main2}. Finally, in Appendix \ref{app1}, we provide a technique for constructing potentials which satisfy the abstract conditions appearing in this article and other literature.\\\\

\section{Preliminaries}\label{S:2}

\subsection{Gagliardo-Nirenberg Inequality}\label{GNQ}

We start with recalling the Gagliardo-Nirenberg inequality \cite[p.168]{BL1976}, \cite[Theorem 1.3.7]{Caz-book},
\be\label{GN}
\norm{u}_{L^{p+1}(\R^n)}^{p+1}\leq C_{GN}\norm{\nabla u}_{L^2(\R^n)}^{\frac{n(p-1)}{2}}\norm{u}_{L^2(\R^n)}^{2-\frac{(n-2)(p-1)}{2}},
\ee
which holds for $ 1 \le p \le \frac{n+2}{n-2}$, which includes the case $0<s_c<1$.

From Kato's diamagnetic inequality \cite{LiebLoss}, we have the following estimate
\be\label{Kato}
\abs{\nabla \abs{u}}\leq \abs{\nabla_A u}.
\ee

Combining \eqref{GN} with \eqref{Kato}, we obtain the following magnetic version
\be\label{GNm}
\norm{u}_{L^{p+1}(\R^n)}^{p+1}\leq C_{GN}\norm{\nabla_A u}_{L^2(\R^n)}^{\frac{n(p-1)}{2}}\norm{u}_{L^2(\R^n)}^{2-\frac{(n-2)(p-1)}{2}}.
\ee
To simplify the notation, we write $\norm{\cdot}_{L^p(\R^n)}$ simply as $\norm{\cdot}_p$. The following equivalent form of the exponents in \eqref{GNm} will be useful 
\be\label{rewriting}
 \norm{u}_{p+1}^{p+1} \le C_{GN} \norm{\grad_A u}_2^2 \left( \norm{\grad_A u}_2^{s_c} \norm{u}_2^{1-s_c} \right)^{p-1}.
 \ee

We are concerned with the case $0<s_c<1$, which we assume in the sequel. The Gagliardo-Nirenberg inequality \eqref{GN} admits a positive smooth $H^1$ minimizer, $Q$, unique up to scaling and translation satisfying (via Pokhozhaev identities, e.g., see \cite{Weinstein83,HR07,RR2022})
\begin{align}
\norm{\grad Q}_2 &= \norm{Q}_2,\\
\norm{Q}_{p+1}^{p+1} &= \frac{p+1}{2} \norm{Q}_2^2,\\
C_{GN}&=\frac{p+1}{2\norm{Q}_2^{p-1}}.
\end{align}
A soliton solution to NLS can be written as
\[u_Q(x,t) = e^{i \lambda t}Q(\alpha x),\]
where
\[\alpha = \frac{\sqrt{n(p-1)}}{2} \quad \quad \mbox{and} \quad \quad \lambda = 1- \frac{(n-2)(p-1)}{4} = (1-s_c)\frac{p-1}2.\]
We have 
\be\label{norms}
\norm{u_Q}_{p+1}^{p+1}=\frac 1{\alpha^n}\norm{Q}_{p+1}^{p+1}=\frac{p+1}{2\alpha^n}\norm{Q}_{2}^{2},
\ee
as well as
\be\label{graduQ}
\norm{\nabla u_Q}_2^2=\frac{\alpha^2}{\alpha^n}\norm{Q}_2^2,
\ee
and
\be\label{EuQ}
E[u_Q]=\frac1{2\alpha^n}({\alpha^2}-1)\norm{Q}_2^2.
\ee
Finally, for future reference we note 
\be\label{useful}
1+\frac{s_c(p-1)}{2}=\frac {n(p-1)}{4}.
\ee

\subsection{Scale Invariant Quantities and Minimal Energy Function}\label{Pdef}

If $u$ solves (emNLS), then 
\[
u_\lambda(x,t)=\lambda^\frac{2}{p-1}u(\lambda x, \lambda^2 t) 
\]
solves (emNLS) with potentials
\[
(A_{0,\lambda}(x,t), A_\lambda(x,t))=(\lambda^2A_0(\lambda x,\lambda^2t),\lambda A(\lambda x,\lambda^2 t)).
\]

We consider the scale invariant quantities
\[\norm{\grad_A u}_2^{2s_c} \norm{u}_2^{2(1-s_c)} \quad \quad \mbox{and} \quad \quad E_A[u]^{s_c} \norm{u}_2^{2(1-s_c)},\]
and similarly to \cite{HR07}, we note that by \eqref{GNm} and \eqref{useful}
\be\label{lowerbounde}
\begin{split}
&E_A[u] \norm{u}_2^{2\frac{(1-s_c)}{s_c}}=\frac{1}{2} \norm{\grad_A u}_2^2 \norm{u}_2^{2\frac{(1-s_c)}{s_c}} - \frac{1}{p+1} \norm{ u}_{p+1}^{p+1} \norm{u}_2^{2\frac{(1-s_c)}{s_c}}\\
&\quad\qquad\qquad \geq \frac{1}{2} \norm{\grad_A u}_2^2 \norm{u}_2^{2\frac{(1-s_c)}{s_c}} - \frac{C_{GN}}{p+1} \norm{\grad_A u}_2^2 \norm{u}_2^{2\frac{(1-s_c)}{s_c}} \left( \norm{\grad_A u}_2^{s_c} \norm{u}_2^{1-s_c} \right)^{p-1}\\
&\quad\qquad\qquad= \frac{1}{2} \norm{\grad_A u}_2^2 \norm{u}_2^{2\frac{(1-s_c)}{s_c}} -\frac{C_{GN}}{p+1}\left(\norm{\grad_Au}_2^{2} \norm{u}_2^{2\frac{(1-s_c)}{s_c}} \right)^{\frac{n(p-1)}{4}}.
\end{split}
\ee

This motivates the definition of the \textit{minimal energy function} (cf. \cite{HR07}, \cite{KM2006})
\be\label{mef}
P(x) = \frac{1}{2}x - \frac{C_{GN}}{p+1} x^{\frac{n(p-1)}{4}}.
\ee
In particular, $P$ is called the minimal energy function due to the bound that follows from \eqref{lowerbounde}
\be\label{Pbound}
P(\norm{\grad_A u}_2^2 \norm{u}_2^{2\frac{(1-s_c)}{s_c}}) \le E_A[u] \norm{u}_2^{2\frac{(1-s_c)}{s_c}}.
\ee
Equality in \eqref{Pbound} holds for the NLS soliton solution $u_Q$, namely, we have
\be\label{GSe}
P( \norm{\grad u_Q}_2^2 \norm{u_Q}_2^{2\frac{(1-s_c)}{s_c}} ) = E[u_Q] \norm{u_Q}_2^{2\frac{(1-s_c)}{s_c}}.
\ee
In addition, the function $P$ is concave down on $[0,\infty)$ and attains the maximum of 
\be\label{max}
E[u_Q] \norm{u_Q}_2^{2\frac{(1-s_c)}{s_c}}
\ee
at 
\be\label{maxp}
x= \norm{\grad u_Q}_2^2 \norm{u_Q}_2^{2\frac{(1-s_c)}{s_c}}.
\ee

Besides the magnetic energy $E_A[u]$, we introduce \emph{Kato energy},
\[
E_K[u]=\frac 12\norm{\nabla \abs{u}}^2_2-\frac 1{p+1}\norm{u}_{p+1}^{p+1}.
\]
Note, due to the positivity of $Q$, we have
\[
E_K[u_Q]=E[u_Q].
\]
Next, just like above, but using the Gagliardo--Nirenberg inequality \eqref{GN} applied to $|u|$ we can show a similar bound to \eqref{Pbound}
\be \label{lowerbound}
 P \left( \norm{\grad |u|}_2^2 \norm{u}_2^{2\frac{(1-s_c)}{s_c}} \right)\leq E_K[u]\norm{u}_2^{2\frac{(1-s_c)}{s_c}}. 
\ee

\subsection{Virial Identity}
The virial identity for the standard NLS is well-known \cite{Vlasov1971, zakharov1972collapse, glassey}. For a general linear emNLS it was derived first in \cite{FanelliVega09}. Including nonlinearity is an easy extension (see e.g. \cite{CCL}), and it can be stated as follows
\be\label{virial}
\begin{split}
\partial_t^2\int \abs{x}^2\abs{u}^2 dx&=4n(p-1)E_A[u]-2(n(p-1)-4)\norm{\nabla_Au}^2-2n(p-1)\int A_0\abs{u}^2dx\\
&\qquad\qquad-4\int_{\R^n}\abs{x} \partial_rA_0 \abs{u}^2\ dx
+8\int_{\R^n}\abs{x}\mbox{Im}(uB_\tau\cdot \overline{\nabla_A u}) \ dx.
\end{split}
\ee
If the electromagnetic potentials are all identically zero, then the third, fourth and fifth terms are not present, and we recover the virial identity for the standard NLS.

We note that even in the purely magnetic equation, an extra term involving the trapping term $B_\tau$ (the last term in \eqref{virial}) remains. Therefore, besides being an obstruction to dispersion, this term, in some sense, can also be an obstruction to virial blow-up arguments.

\subsection{Magnetic Hardy's inequality}
 We recall a magnetic Hardy's inequality (see e.g. \cite[A.1]{FanelliVega09}). If $u\in H^1_A(\mathbb{R}^n)$ and $n\geq 3$, then
 \be\label{mH}
 \norm{\frac u{\abs{x}}}_2\leq \frac 2{n-2}\norm{\nabla_A u}_2.
 \ee
 \subsection{Conditions on the Potentials}
If we include the electric potential, we can still work with the same norm for $B_\tau$ as before, and then we also need bounds on the electric potential and its derivative. More specifically, in this setting we consider potentials $A,A_0$ so that, with notation $f_{-} = \mbox{max} \{ -f,0\}$, the estimates
\begin{align}
\int A_{0,-} \abs{u}^2 dx\leq a_0 \norm{\nabla_A u}^2_2,\label{a0}\\
\int_{\R^n}(\partial_rA_0)_{-}\abs{x} \abs{u}^2\ dx\leq a_1\norm{\nabla_A u}^2,\label{a1}\\
\int_{\R^n}\abs{x}\mathcal Im(uB_\tau\cdot \overline{\nabla_A u}) \ dx \leq b_0 \norm{\nabla_A u}^2_2\label{b0},
\end{align}
hold for all $u \in H^1_{A}(\mathbb{R}^N)$ with constants $0 \le a_0 <1$, and $0 \le a_1,b_0 < \infty$.

As mentioned before, we are interested in the potentials considered in the article \cite{JFA}, as well as the conditions on potentials in \cite{FanelliVega09}. The conditions in \cite{FanelliVega09, JFA} are imposed on the potentials in order to obtain weak dispersion and Strichartz estimates. By \cite[Lemma 2.3]{JFA}, we have $0<a_0<1$.

Next, using \eqref{mH}, we can take
 \be
 b_0=\norm{\abs{x}^2B_\tau}_\infty \frac 2{n-2}.
 \ee
 
 The condition on $B_\tau$ for $n\geq 4$ in \cite{FanelliVega09, JFA} is 
 \be
  \norm{\abs{x}^{2}B_{\tau}}^{2}_{L^{\infty}_{x}}+2\norm{\abs{x}^{3}(\partial_{r}A_{0})_{+}}_{L^{\infty}_{x}}<\frac 23(n-1)(n-3).
 \ee
Hence in that case, 
\[
b_0\leq \frac 43\frac{(n-1)(n-3)}{n-2}.
\]
This can be compared with bounds in \eqref{E:trap1}. We discuss connections between these conditions and those necessary for Strichartz estimates in 3D in Appendix \ref{app1}.

Finally, having nonzero $(\partial_r A_0)_-$ is the place where we need to impose an additional condition that does not appear in the work of \cite{FanelliVega09, JFA}. There, a smallness condition on $(\partial_r A_0)_+$ is needed. We use the condition \eqref{a1} to establish a finite time blow-up. Therefore, if we are only after global wellposedness, the condition \eqref{a1} would not be needed (see Corollary \ref{cor1}).
\section{Main Lemmas}\label{S:3}
There are several technical lemmas we need to establish.

\subsection{Remainder Identity for Kato's Inequality}
Here we establish a remainder identity for Kato's inequality on $H^1_A$. 
 
\begin{lemma}[Remainder Identity for Kato's Inequality]\label{remainder}
For $u \in H^1_A$, the identity
\begin{align*}
\norm{\grad_A u}_2^2 - \norm{ \grad |u|}_2^2 &= \int_{u \neq 0} \Big| \mbox{Im} \Big[ \frac{\overline{u}}{|u|} \grad_A u \Big]\Big|^2 dx
\end{align*}
is satisfied.
\end{lemma}
\begin{proof}
Given $u \in H^1_A$, we have pointwise \cite[Thm 6.17]{LiebLoss}
\begin{align*}
|\partial_j |u||^2 &= \begin{cases} \left|\mbox{Re} \Big[ \frac{\overline{u}}{|u|} \partial_j u \Big]\right|^2 & u \neq 0\\
0 &u = 0 \end{cases}\\
&= \begin{cases} \left| \mbox{Re} \Big[ \frac{\overline{u}}{|u|} D_j u \Big] \right|^2 & u \neq 0\\
0 & u = 0 . \end{cases}
\end{align*}
Furthermore by \cite[Thm 6.19]{LiebLoss} we have that a.e.
\[
\Big| D_j u \Big|^2 = \begin{cases} \Big| \frac{\overline{u}}{\abs{u}} D_j u \Big|^2 & u \neq 0\\ 0 & u=0. \end{cases}
\]
It follows that a.e.
\[ 
|\grad_A u|^2 - |\grad |u||^2 =  \begin{cases} \Big|  \mbox{Im} \Big[ \frac{\overline{u}}{|u|} \nabla_Au \Big] \Big|^2 & u \neq 0\\ 0 & u=0.
\end{cases}
\]
From here the result follows by integration, which completes the proof.
\end{proof}

\subsection{Double Cut Lemma}
We show that the set of the initial data strictly below the ground state is the union of two disjoint open sets such that within each set we can also compare the size of $\norm{\nabla_Au}_2$ to $\norm{\nabla\abs{u}}_2$.

\begin{lemma}[Double Cut Lemma]\label{lem:double-cut}
For a potential $A \in L^2_{\text{loc}}(\mathbb{R}^n;\mathbb{R}^n)$, consider
\[
\mathcal{R} = \{u \in H^1_A(\mathbb{R}^N): E_A[u]^{s_c}\norm{u}_2^{2(1-s_c)}<E[u_Q]^{s_c}\norm{u_Q}_2^{2(1-s_c)}\}.
\]
There are two disjoint sets $\mathcal{R}_1,\mathcal{R}_2$ with $\mathcal{R}=\mathcal{R}_1 \cup \mathcal{R}_2$ such that for each $u \in \mathcal{R}_1$,
\begin{equation}\label{E:R1}
\norm{\grad|u|}_2^{s_c} \norm{u}_2^{1-s_c} \le \norm{\grad_A u}_2^{s_c} \norm{u}_2^{1-s_c} < \norm{\grad u_Q}_2^{s_c} \norm{u_Q}_2^{1-s_c},
\end{equation}
and for each $u \in \mathcal{R}_2$,
\begin{equation}\label{E:R2}
\norm{\grad u_Q}_2^{s_c} \norm{u_Q}_2^{1-s_c} < \norm{\grad|u|}_2^{s_c} \norm{u}_2^{1-s_c} \le \norm{\grad_A u}_2^{s_c} \norm{u}_2^{1-s_c}.
\end{equation}
\end{lemma}

\begin{proof}
\par Recall the minimal energy function from Section \ref{Pdef}
\[P(x) = \frac{1}{2} x - \frac{C_{GN}}{p+1} x^{\frac{n(p-1)}{4}},\]
where for $u \in H^1_A$, the inequality
\[P( \norm{\grad_A u}_2^2 \norm{u}_2^{2\frac{(1-s_c)}{s_c}} ) \le E_A[u] \norm{u}_2^{2\frac{(1-s_c)}{s_c}}\]
holds by \eqref{lowerbounde}. Now, since for $u\in \mathcal{R}$ we have
\[
E_A[u] \norm{u}_2^{2\frac{(1-s_c)}{s_c}} < E[u_Q] \norm{u_Q}_2^{2\frac{(1-s_c)}{s_c}},
\]
and by \eqref{GSe}
\[P(\norm{\grad u_Q}_2^2 \norm{u_Q}_2^{2\frac{(1-s_c)}{s_c}}) = E[u_Q] \norm{u_Q}_2^{2\frac{(1-s_c)}{s_c}},\]
it must be that if $u\in \mathcal{R}$, we either have
\begin{equation}\label{E:R1a}
\norm{\grad_A u}_2^{s_c} \norm{u}_2^{1-s_c} < \norm{\grad u_Q}_2^{s_c} \norm{u_Q}_2^{1-s_c}
\end{equation}
%\quad\mbox{
or
%}\quad
\begin{equation}\label{as2}%{E:R2a}
 \norm{\grad_A u}_2^{s_c} \norm{u}_2^{1-s_c} > \norm{\grad u_Q}_2^{s_c} \norm{u_Q}_2^{1-s_c}.
\end{equation}
We now show that we can further describe the set $\mathcal{R}$ as a union of $\mathcal{R}_1$ and $\mathcal{R}_2$. By Kato's inequality, to establish the full result it is sufficient to show that each $u$ satisfying \eqref{as2}
also satisfies
\be\label{wts}
\norm{\grad |u|}_2^{s_c} \norm{u}_2^{1-s_c} > \norm{\grad u_Q}_2^{s_c} \norm{u_Q}_2^{1-s_c}.
\ee
Suppose we have $u \in H^1_A$ such that
\[
E_A[u]^{s_c}\norm{u}_2^{2(1-s_c)}<E[u_Q]^{s_c}\norm{u_Q}_2^{2(1-s_c)},
\]
and \eqref{as2}.
%\be\label{as2}
% \norm{\grad_A u}_2^{s_c} \norm{u}_2^{1-s_c} > \norm{\grad u_Q}_2^{s_c} \norm{u_Q}_2^{1-s_c}.
% \ee
Then 
\[
E_A[u] \norm{u}_2^{2\frac{(1-s_c)}{s_c}} < E[u_Q]\norm{u_Q}_2^{2\frac{(1-s_c)}{s_c}},
\]
which with \eqref{as2} %then 
implies
\[
 \norm{u}_{p+1}^{p+1}\norm{u}_2^{2\frac{(1-s_c)}{s_c}} >   \norm{u_Q}_{p+1}^{p+1}\norm{u_Q}_2^{2\frac{(1-s_c)}{s_c}}.
\]

Applying the Gagliardo-Nirenberg inequality \eqref{GN} on the left hand side and using that $Q$ is the optimizer, and so is $u_Q$ by scale invariance (i.e., applying equality in \eqref{GN} to the right-hand side above), we obtain

\begin{align*}
{C_{GN}}\norm{\grad |u|}_{2}^{\frac{n(p-1)}{2}}\norm{u}_2^{2-\frac{(n-2)(p-1)}{2}} \norm{u}_2^{2\frac{(1-s_c)}{s_c}} > {C_{GN}}\norm{\grad u_Q}_{2}^{\frac{n(p-1)}{2}}\norm{u_Q}_2^{2-\frac{(n-2)(p-1)}{2}} \norm{u_Q}_2^{2\frac{(1-s_c)}{s_c}}.
\end{align*}
Then rewriting the exponents, as in \eqref{rewriting}, gives
\begin{align*}
  \norm{\grad |u|}_{2}^{2} \left(\norm{\grad |u|}_2^{s_c} \norm{u}_2^{1-s_c} \right)^{p-1} \norm{u}_2^{2\frac{(1-s_c)}{s_c}} >  \norm{\grad u_Q}_{2}^{2} \left(\norm{\grad u_Q}_2^{s_c} \norm{u_Q}_2^{1-s_c} \right)^{p-1} \norm{u_Q}_2^{2\frac{(1-s_c)}{s_c}}.
  \end{align*}
  Manipulating the exponents further, we see the above is exactly
 \begin{align*}
  \left(\norm{\grad |u|}_2^{s_c} \norm{u}_2^{1-s_c} \right)^{(p-1) + \frac{2}{s_c}}  >   \left(\norm{\grad u_Q}_2^{s_c} \norm{u_Q}_2^{1-s_c} \right)^{(p-1) + \frac{2}{s_c}} ,
  \end{align*}
  which in turn implies \eqref{wts} as claimed.

\end{proof}

\subsection{Proximity Estimate for Kinetic Energy Above the Ground State}
In this section, we establish the main estimate, which will be used in our blow-up argument.

For $u \in \mathcal{R}_2$, as defined in Lemma \ref{lem:double-cut}, we consider two non-negative quantities, which are useful for measuring how far $u$ is from being a ground state. The first quantity arises from the remainder $u$ in Kato's inequality and is given by
\[\mathcal{F}_{RK}(u) \defneq \norm{\grad_A u}_2^2\norm{u}_2^{2\frac{(1-s_c)}{s_c}} -  \norm{\grad |u|}_2^2\norm{u}_2^{2\frac{(1-s_c)}{s_c}}.\]
The second quantity arises from the kinetic energy of $u$ and is given by
\[\mathcal{F}_{KE}(u) \defneq  \norm{\grad_A u}_2^2\norm{u}_2^{2\frac{(1-s_c)}{s_c}} - \norm{\grad u_Q}_2^2\norm{u_Q}_2^{2\frac{(1-s_c)}{s_c}} .\]
The Double Cut Lemma \ref{lem:double-cut} equation \eqref{E:R2} implies the bound
\be\label{1stbound}
\mathcal{F}_{RK}(u) \leq \mathcal{F}_{KE}(u),
\ee
when $u \in \mathcal{R}_2$. In fact, what the lemma below shows is that the geometry of $\mathcal{R}_2$ additionally provides a bound
\[\mathcal{F}_{RK}(u) \lesssim \left(\mathcal{F}_{KE} (u) \right)^2,\]
which holds on $\mathcal{R}_2$ and is stronger than the linear bound \eqref{1stbound} when $\mathcal{F}_{KE}$ is small.

\begin{lemma}[Ground State Proximity Bound]\label{lem:GSP}
There exists a constant $C_{p,n}>0$ such that 
\begin{equation}\label{GSPeq}
\left(\mathcal{F}_{RK}(u)\right)^{1/2} \le C_{p,n} \, \mathcal{F}_{KE}(u)
\end{equation}
for $u \in \mathcal{R}_2$.\\
\par Furthermore, the inequality holds for the following values of $C_{p,n}$.
\begin{itemize}
\item For $n(p-1) \le 8$,
\[C_{p,n}=\left( \frac{16(p+1)}{C_{GN} n(p-1)\left(n(p-1) - 4\right)}\right)^{-1/2} \left( \norm{\grad u_Q }_2^2\norm{u_Q}_2^{2\frac{(1-s_c)}{s_c}}  \right)^{\frac{n(p-1) - 8}{8}}.\]
\item For $n(p-1) > 8$, 
\begin{align*}
C_{p,n} & = 1 + \left( \frac{16(p+1)}{C_{GN} n(p-1)\left(n(p-1) - 4\right)}\right)^{-1/2} \left( 1+  \norm{\grad u_Q }_2^2\norm{u_Q}_2^{2\frac{(1-s_c)}{s_c}}  \right)^{ \frac{n(p-1) - 8}{8}}.
\end{align*}
\end{itemize}
\end{lemma}
\par Before we begin the proof of Lemma \ref{lem:GSP} we make some remarks. We also state a simple Calculus lemma that is used in the proof. 

The bound of Lemma \ref{lem:GSP} arises from the geometry of $\mathcal{R}_2$. The minimal energy function $P$ introduced in Section \ref{Pdef} \eqref{mef} provides a description of the relevant geometry. Prior to establishing the desired bound, we define a rigid transformation of the minimal energy function into a form better suited to this context.
\par Recall, $P(x)$ is concave down on $[0,\infty)$, attaining its maximum of 
$E[u_Q] \norm{u_Q}_2^{2\frac{(1-s_c)}{s_c}}$ at $x_\ast= \norm{\grad u_Q}_2^2 \norm{u_Q}_2^{2\frac{(1-s_c)}{s_c}}$. We define a transformation, $P^*$, which recenters the curve $P$ to the ground state.
\begin{align*}
P^*(x) &:=P(x_\ast) -P\left( x+ x_\ast\right) \\
&=E[u_Q]\norm{u_Q}_2^{2\frac{(1-s_c)}{s_c}}  -P\left( x+  \norm{\grad u_Q }_2^2\norm{u_Q}_2^{2\frac{(1-s_c)}{s_c}} \right).\\ 
\end{align*}
This transformation is defined so that $P^*$ is concave up on $[-x_\ast,\infty)$ and increasing on $[0,\infty)$, and such that $P^*(0)=(P^*)'(0) = 0$. 

We now recall the bound \eqref{lowerbound}, namely,
\[
E_K[u]\norm{u}_2^{2\frac{(1-s_c)}{s_c}} \ge P \left( \norm{\grad |u|}_2^2 \norm{u}_2^{2\frac{(1-s_c)}{s_c}} \right).
\]
Under the transformation of $P$ to $P^{\ast}$, this bound implies
\[ 
P(x_\ast)-  E_K[u] \norm{u}_2^{2\frac{(1-s_c)}{s_c}}\le P^*\left(\norm{\grad |u|}_2^2\norm{u}_2^{2\frac{(1-s_c)}{s_c}}  - x_\ast\right),
\]
and applying the main energy assumption \eqref{eless} to the left hand side, with the explicit value of $P(x_{\ast})$, results in the bound
\begin{equation}\label{pbound}
\frac{1}{2}\norm{u}_2^{2\frac{(1-s_c)}{s_c}} \left[ \norm{\grad_A u}_2^2 - \norm{\grad |u|}_2^2\right] \le P^*\left(\norm{\grad |u|}_2^2\norm{u}_2^{2\frac{(1-s_c)}{s_c}}  - x_\ast\right),
\end{equation}
which by \eqref{Kato} and the definition of $\mathcal F_{RK}(u)$ and $\mathcal F_{KE}(u)$ gives
\be\label{pbound2}
 \frac{1}{2}\mathcal{F}_{RK}(u)\leq P^\ast(\mathcal F_{KE}(u)).
\ee
\par We now consider the inverse of $P^*$ on $[0,\infty)$, denoted by $(P^*)^{-1}$. Since $P^*$ is strictly increasing in $[0,\infty)$, the inverse $(P^*)^{-1}$ is well defined and increasing. Applying $\left(P^{\ast}\right)^{-1}$ to both sides of \eqref{pbound2} results in
\begin{equation}\label{inversebound}
(P^*)^{-1}\left( \frac{1}{2}\mathcal{F}_{RK}(u) \right) \le \mathcal{F}_{KE} (u).
\end{equation}
 
To establish the Lemma \ref{lem:GSP}, we use a simple Calculus fact that we state without proof.
\begin{lemma}\label{Calc}
Let $b\in (0,\infty]$, $f:[0,b)\to [0,\infty)$ be twice differentiable with well-defined inverse, $f^{-1},$ and such that $f(0)=f'(0)=0$. If $f''(x)\leq C$ for some $C>0$ for all $x\in[0,b)$, then $f(x)\leq \frac C2 x^2$ for all $x\in[0,b)$, and $f^{-1}(x)\geq\sqrt{\frac 2C}\sqrt{x}$ for all $x\in [0,f(b))$.
\end{lemma}

We are now ready to start the proof.
\begin{proof}[Proof of Lemma \ref{lem:GSP}]
We consider three cases. 

\noindent
{\textbf{Case 1: $n(p-1)=8$.}} 
In this case, by completing the square, we can see $P^*:[0,\infty)\to [0,\infty)$ is simply a quadratic function 
\[
P^\ast(x)=\frac{C_{GN}}{p+1}x^2,
\]
so
\[
(P^\ast)^{-1}(x)=\sqrt{\frac{p+1}{C_{GN}}}\sqrt{x},
\]
and \eqref{inversebound} implies
\[
\sqrt{\frac{p+1}{2C_{GN}}}\left(\mathcal{F}_{RK}(u)\right)^\frac 12 \le \mathcal{F}_{KE} (u),
\]
which gives \eqref{GSPeq} as needed.

\noindent
{\textbf{Case 2: $n(p-1)<8$.}}
We begin by computing 
\begin{align*}
\partial_x^2 P^*(x) &= - \partial_x^2 P \left(x+  \norm{\grad u_Q }_2^2 \norm{u_Q}_2^{2\frac{(1-s_c)}{s_c}} \right)\\
&= \frac{C_{GN}}{p+1} \left( \frac{n(p-1)}{4} \right) \left( \frac{n(p-1)}{4} - 1 \right) \left(x+  \norm{\grad u_Q }_2^2 \norm{u_Q}_2^{2\frac{(1-s_c)}{s_c}} \right)^{\frac{n(p-1) - 8}{4}}\\
&\leq \frac{C_{GN}}{p+1} \left( \frac{n(p-1)}{4} \right) \left( \frac{n(p-1)}{4} - 1 \right) \left(  \norm{\grad u_Q }_2^2 \norm{u_Q}_2^{2\frac{(1-s_c)}{s_c}} \right)^{\frac{n(p-1) - 8}{4}},
\end{align*}
for $x\in [0,\infty).$ Then by Lemma \ref{Calc} 
    \begin{align*}
    (P^*)^{-1}(x) &\ge \left(\frac{2(p+1)}{C_{GN} \left( \left(\frac{n(p-1)}{4}\right)^2 - \frac{n(p-1)}{4} \right)\left(  \norm{\grad u_Q }_2^2 \norm{u_Q}_2^{2\frac{(1-s_c)}{s_c}}\right)^{\frac{n(p-1) - 8}{4}}}\right)^{1/2} x^{1/2}\\
    &= \sqrt{2} \left[ \left( \frac{16(p+1)}{C_{GN} n(p-1)\left(n(p-1) - 4\right)}\right)^{-1/2} \left(  \norm{\grad u_Q }_2^2\norm{u_Q}_2^{2\frac{(1-s_c)}{s_c}} \right)^{ \frac{n(p-1) - 8}{8}} \right]^{-1} x^{1/2},
    \end{align*}
    and \eqref{inversebound}
implies \eqref{GSPeq} in this case.

\noindent
{\textbf{Case 3: $n(p-1)>8$.}}
In this case, we consider $\mathcal{F}_{KE} (u)< 1$ and $\mathcal{F}_{KE} (u) \ge 1$.  If $\mathcal{F}_{KE} (u) \ge 1$, then \eqref{1stbound} implies \eqref{GSPeq}.

Next, if $\mathcal{F}_{KE} (u) < 1$, we observe that in $[0,1)$ we have
\begin{align*}
\partial_x^2 P^*(x) &= - \partial_x^2 P \left(x+  \norm{\grad u_Q }_2^2 \norm{u_Q}_2^{2\frac{(1-s_c)}{s_c}} \right)\\
&\leq \frac{C_{GN}}{p+1} \left( \frac{n(p-1)}{4} \right) \left( \frac{n(p-1)}{4} - 1 \right) \left(1+  \norm{\grad u_Q }_2^2 \norm{u_Q}_2^{2\frac{(1-s_c)}{s_c}} \right)^{\frac{n(p-1) - 8}{4}},
\end{align*}
and by Lemma \ref{Calc} for all $x\in [0,P^\ast(1))$
 
  \begin{align*}
    (P^*)^{-1}(x) &\ge \sqrt{2} \left[ \left( \frac{16(p+1)}{C_{GN} n(p-1)\left(n(p-1) - 4\right)}\right)^{-1/2} \left( 1+  \norm{\grad u_Q }_2^2\norm{u_Q}_2^{2\frac{(1-s_c)}{s_c}}  \right)^{ \frac{n(p-1) - 8}{8}}\right]^{-1} x^{1/2}.
    \end{align*}
We now note that since $\mathcal{F}_{KE} (u) < 1$, by \eqref{pbound2}, $\frac 12 \mathcal F_{RK}(u)\in[0,P^\ast(1))$, so by Lemma \ref{Calc}, we obtain

\begin{align*}
(P^*)^{-1} &\left(\frac{1}{2} \mathcal{F}_{RK}(u)\right) \ge \\ & \left[ \left( \frac{16(p+1)}{C_{GN} n(p-1)\left(n(p-1) - 4\right)}\right)^{-1/2} \left( 1+  \norm{\grad u_Q }_2^2\norm{u_Q}_2^{2\frac{(1-s_c)}{s_c}}  \right)^{ \frac{n(p-1) - 8}{8}}\right]^{-1} \left( \mathcal{F}_{RK}(u) \right)^{1/2}.
\end{align*}
Hence, using \eqref{inversebound} as above,  we obtain
\begin{align*}
\left( \mathcal{F}_{RK} (u)  \right)^{1/2}& \,\le\, \left[\left( \frac{16(p+1)}{C_{GN} n(p-1)\left(n(p-1) - 4\right)}\right)^{-1/2} \left(1 +  \norm{\grad u_Q }_2^2\norm{u_Q}_2^{2\frac{(1-s_c)}{s_c}}  \right)^{\frac{n(p-1) - 8}{8}} \right]\,\mathcal{F}_{KE}(u)\\
&\le\, \left[1+ \left( \frac{16(p+1)}{C_{GN} n(p-1)\left(n(p-1) - 4\right)}\right)^{-1/2} \left(1 + \norm{\grad u_Q }_2^2\norm{u_Q}_2^{2\frac{(1-s_c)}{s_c}}   \right)^{\frac{n(p-1) - 8}{8}} \right]\,\mathcal{F}_{KE}(u),
\end{align*}
as claimed.
\end{proof}

\subsection{Bound for the Term Involving Trapping}
The following lemma builds on the previous lemmas and is a key component of the proof of the blow-up.
\begin{lemma}\label{btaubound}
If $\norm{\abs{x}^2 B_\tau}_\infty$ is small enough, and $u\in \mathcal{R}_2$, then
\[8 \norm{u}_2^{2\frac{(1-s_c)}{s_c}}\, \mbox{Im} \int |x|uB_{\tau} \cdot \overline{\grad_A u}\ dx \le\, 2(n(p-1)-4) \, \mathcal{F}_{KE}(u).\]
\end{lemma}

\begin{proof}
Let $u\in \mathcal R_2$, then $\mathcal{F}_{KE}(u)> 0$.  We can consider the integral over the union of the sets where $u=0$ and when $u\neq 0$. If the integral is taken over the set where $u=0$, then the inequality holds since we are considering the intercritical case.  Therefore, we can assume $u\neq 0$.  We proceed to bound the integral over that set.  Since $B_\tau$ is real-valued we have
\begin{align*}
\abs{ \mbox{Im} \int |x|uB_{\tau} \cdot \overline{\grad_A u}\ dx}
&=\abs{ \int \abs{x}B_\tau\cdot \mbox{Im}\left(  u\overline{\grad_A u}\right) \ dx}\\
&\leq \norm{\abs{x}^2 B_\tau}_\infty \norm{\frac u{\abs{x}}}_2 \norm{ \mbox{Im}\left( \frac {u}{\abs{u}}\overline{\grad_A u}\right)}_2\\
&\leq \frac 2{n-2}\norm{\abs{x}^2 B_\tau}_\infty \norm{\nabla_Au}_2 \norm{ \mbox{Im}\left( \frac {u}{\abs{u}}\overline{\grad_A u}\right)}_2,
\end{align*}
by Hardy's inequality \eqref{mH}.  Then by Lemma \ref{remainder} and \eqref{1stbound} we have
\begin{align*}
& \norm{u}_2^{2\frac{(1-s_c)}{s_c}} \left| 8\mbox{Im} \int |x|uB_{\tau} \cdot \overline{\grad_A u}\ dx \right|\le \norm{u}_2^{\frac{(1-s_c)}{s_c}} \frac{16}{n-2} \norm{|x|^2 B_{\tau}}_{\infty} \norm{\nabla_Au}_2  \left[ \mathcal{F}_{RK}(u)\right]^{1/2}\\
& = \frac{16}{n-2} \norm{|x|^2 B_{\tau}}_{\infty} \left[\norm{u_Q}_2^{2\frac{(1-s_c)}{s_c}}\norm{\grad u_Q}_2^2  + \norm{u}_2^{2\frac{(1-s_c)}{s_c}}\norm{\nabla_A u}_2^2 - \norm{u_Q}_2^{2\frac{(1-s_c)}{s_c}}\norm{\grad u_Q}_2^2 \right]^{1/2}  \left[\mathcal{F}_{RK}(u)\right]^{1/2}\\
& \le \frac{16}{n-2} \norm{|x|^2 B_{\tau}}_{\infty} \,\left[\norm{u_Q}_2^{\frac{(1-s_c)}{s_c}} \norm{\grad u_Q}_2\right]\, \left[\mathcal{F}_{RK}(u) \right]^{1/2} + \frac{16}{n-2} \norm{|x|^2 B_{\tau}}_{\infty} \mathcal{F}_{KE}(u)\\
&= I + II.
\end{align*}
To complete the proof we show
\[I \le (n(p-1)-4) \mathcal{F}_{KE}(u),\]
and
\[II \le (n(p-1)-4) \mathcal{F}_{KE}(u).\]

By Lemma \ref{lem:GSP} there is a constant $C_{p,n}$ such that
\[\left[\mathcal{F}_{RK}(u)\right]^{1/2} \le C_{p,n}\, \mathcal{F}_{KE}(u).\]
It follows
\begin{align*}
I &= \frac{16}{n-2} \norm{|x|^2 B_{\tau}}_{\infty} \,\left[\norm{u_Q}_2^{\frac{(1-s_c)}{s_c}} \norm{\grad u_Q}_2\right]\, \left[\mathcal{F}_{RK}(u) \right]^{1/2}\\
& \le \left[ \frac{16\, C_{p,n}}{(n-2)(n(p-1)-4)}\norm{u_Q}_2^{\frac{(1-s_c)}{s_c}} \norm{\grad u_Q}_2 \norm{|x|^2 B_{\tau}}_{\infty} \right] \, \cdot \, \left[ (n(p-1)-4) \mathcal{F}_{KE}(u) \right].
\end{align*}
So we obtain the desired bound if
\[\norm{|x|^2 B_{\tau}}_{\infty} \le \frac{(n-2)(n(p-1)-4)}{16\, C_{p,n}\norm{u_Q}_2^{\frac{(1-s_c)}{s_c}} \norm{\grad u_Q}_2}.\]
Next, if
\[\norm{|x|^2 B_{\tau}}_{\infty} \le \frac{(n-2)(n(p-1)-4)}{16},\]
then
\begin{align*}
II = \frac{16}{n-2} \norm{|x|^2 B_{\tau}}_{\infty} \mathcal{F}_{KE}(u)\le (n(p-1)-4) \mathcal{F}_{KE}(u),
\end{align*}
as needed.
\end{proof}

\section{Proof of Theorem \ref{main}}\label{S:4}

We break the proof of the theorem into two main parts. We begin with the global existence.
\subsection{Proof of Theorem \ref{main}: Global Existence}
 The reasoning for global existence is the same as in \cite{HR07}. The details are as follows. The conservation of mass and energy, \eqref{Pbound}, and \eqref{eless} imply for all time
 \[
P(\norm{\grad_A u(t)}_2^2 \norm{u(t)}_2^{2\frac{(1-s_c)}{s_c}}) \leq E_A[u(t)] \norm{u(t)}_2^{2\frac{(1-s_c)}{s_c}}<E[u_Q] \norm{u_Q}_2^{2\frac{(1-s_c)}{s_c}},
\]
and hence,
\begin{equation}\label{eq:nonequality}
\norm{\grad_A u(t)}_2^2 \norm{u(t)}_2^{2\frac{(1-s_c)}{s_c}} \neq \norm{\grad u_Q}_2^2 \norm{u_Q}_2^{2\frac{(1-s_c)}{s_c}}.
\end{equation}
Then, if \eqref{nablaless} holds, by continuity of $\norm{\nabla_A u(t)}_2$ in $t$, we must have \eqref{nablalesst}. Then the standard blow-up criterion implies that the solution must exist for all time.

\subsection{Proof of Theorem \ref{main}: Blow-Up}
Now, suppose \eqref{nablamore} holds. Then \eqref{nablamoret} holds by \eqref{eq:nonequality} and the continuity of $\norm{\nabla_A u(t)}_2$.  
Next, from the virial identity \eqref{virial} we get
\begin{align*}
\norm{u}_2^{2\frac{(1-s_c)}{s_c}}\, \partial_t^2 \int |x|^2 |u|^2 dx&= 4n(p-1) E_A[u]\norm{u}_2^{2\frac{(1-s_c)}{s_c}} - 2(n(p-1)-4)\norm{\grad_A u}_2^2\norm{u}_2^{2\frac{(1-s_c)}{s_c}}\\
&\quad +\norm{u}_2^{2\frac{(1-s_c)}{s_c}}\,8\mbox{Im} \int |x|uB_{\tau} \cdot \overline{\grad_A u} dx.
\end{align*}
From \eqref{graduQ} and \eqref{EuQ} we have
\be\label{equal}
4n(p-1)E[u_Q]\norm{u_Q}_2^{2\frac{(1-s_c)}{s_c}} - 2(n(p-1)-4)\norm{\grad u_Q}_2^2\norm{u_Q}_2^{2\frac{(1-s_c)}{s_c}}\,=0,
\ee
which allows us to rephrase the virial identity as
\begin{align*}
\norm{u}_2^{2\frac{(1-s_c)}{s_c}}\, \partial_t^2 \int |x|^2 |u|^2 dx&= 4n(p-1)\left(E_A[u] \norm{u}_2^{2\frac{(1-s_c)}{s_c}} - E[u_Q]\norm{u_Q}_2^{2\frac{(1-s_c)}{s_c}}\right)\\
 & \quad - 2(n(p-1)-4) \left( \norm{\grad_A u}_2^2\norm{u}_2^{2\frac{(1-s_c)}{s_c}}-  \norm{\grad u_Q}_2^2\norm{u_Q}_2^{2\frac{(1-s_c)}{s_c}}\right)\\
&\quad +8 \norm{u}_2^{2\frac{(1-s_c)}{s_c}}\, \mbox{Im} \int |x|uB_{\tau} \cdot \overline{\grad_A u}  dx.
\end{align*}
From \eqref{eless} we have there exists $\varepsilon>0$ such that for all time
\[4n(p-1)\left( \norm{u}_2^{2\frac{(1-s_c)}{s_c}}\, E[u] -\norm{u_Q}_2^{2\frac{(1-s_c)}{s_c}}\, E[u_Q] \right) <-\varepsilon.\]
Then by definition of $\mathcal F_{KE}(u)$, it follows 
\[\norm{u}_2^{2\frac{(1-s_c)}{s_c}}\, \partial_t^2 \int |x|^2 |u|^2  dx< 8\norm{u}_2^{2\frac{(1-s_c)}{s_c}}\, \mbox{Im} \int |x|uB_{\tau} \cdot \overline{\grad_A u} dx - 2(n(p-1)-4) \, \mathcal{F}_{KE}(u) - \varepsilon.\]
From Lemma \ref{btaubound} we have
\[8\norm{u}_2^{2\frac{(1-s_c)}{s_c}}\, \mbox{Im} \int |x|uB_{\tau} \cdot \overline{\grad_A u}\, \le\, 2(n(p-1)-4) \, \mathcal{F}_{KE}(u),\]
so then the blow-up follows by the classical virial identity argument.
 
\section{Proofs With the Electric Potential}\label{S:5}

In this section we prove Theorem \ref{main2}.  We consider the most interesting case, namely, when $A_{0,-}\neq 0$, and $(\partial_r A_0)_-\neq 0$.  If $A_{0,-}= 0$, and $(\partial_r A_0)_-= 0$, the statements of the theorems, and the proofs would be essentially identical to the purely magnetic case. It will be clear from the proof below why this is the case.

The main ideas are as follows.  To obtain global existence, we bound the energy similarly  as in \eqref{lowerbounde}

\be\label{gather}
\begin{split}
  E_A[u]  \norm{u}_2^{2\frac{(1-s_c)}{s_c}}&\geq \frac 12\norm{\nabla_A u}_2^2 \norm{u}_2^{2\frac{(1-s_c)}{s_c}} -\frac{C_{GN}}{p+1}\left(\norm{\grad_Au}_2^{2} \norm{u}_2^{2\frac{(1-s_c)}{s_c}} \right)^{\frac{n(p-1)}{4}} \\
 &\quad -\frac 12 \norm{u}_2^{2\frac{(1-s_c)}{s_c}} \int A_{0,-}\abs{u}^2 dx \\
 &\geq \frac 12\norm{\nabla_A u}_2^2 \norm{u}_2^{2\frac{(1-s_c)}{s_c}}- \frac{C_{GN}}{p+1}\left(\norm{\grad_Au}_2^{2} \norm{u}_2^{2\frac{(1-s_c)}{s_c}} \right)^{\frac{n(p-1)}{4}}\\
 &\quad- \frac{a_{0}}{2}\norm{\nabla_A u}^2_2 \norm{u}_2^{2\frac{(1-s_c)}{s_c}},
 \end{split}
 \ee
 by \eqref{a0}.  Based on this, we modify the definition of the minimal energy function to be
 \[
 F(x)=\frac 12(1-a_0)x-\frac{C_{GN}}{p+1}x^\frac{n(p-1)}{4}.
 \]
 A computation shows that $F$ achieves the maximum at $(1-a_0)^\frac{2}{s_c(p-1)}\norm{\grad u_Q}_2^2 \norm{u_Q}_2^{2\frac{(1-s_c)}{s_c}}$ with the value of $(1-a_0)^{1+\frac 2{s_c(p-1)}}E[u_Q] \norm{u_Q}_2^{2\frac{(1-s_c)}{s_c}}$ (c.f. \eqref{maxp} and \eqref{max}, respectively).
 
 It follows that if \eqref{eless2} and \eqref{nablaless2} hold (or if \eqref{eless4} and \eqref{nablaless4} hold), then again by the continuity of $\norm{\nabla_A u(t)}_2$ in $t$, \eqref{nablaless2t} (and \eqref{nablaless4t}) holds for all time, and the solution exists globally. This concludes the proof of global existence for Theorem \ref{main2} and the proof of Corollary \ref{cor1}.

 \subsection{Proof of Blow-Up in Theorem \ref{main2}}
 
 On the other hand, if \eqref{nablamore2} holds, then again by continuity and \eqref{gather}, \eqref{nablamore2t} follows.  To obtain blowup for finite variance, and given the smallness condition \eqref{c1}, the virial identity, \eqref{virial}, gives us
 \be
\begin{split}
\norm{u}_2^{2\frac{(1-s_c)}{s_c}}\, \partial_t^2 \int |x|^2 |u|^2 dx &=4n(p-1)E_A[u]\norm{u}_2^{2\frac{(1-s_c)}{s_c}}-2(n(p-1)-4)\norm{\nabla_Au}^2_2\norm{u}_2^{2\frac{(1-s_c)}{s_c}}\\
&\qquad -2n(p-1)\int A_0\abs{u}^2dx\norm{u}_2^{2\frac{(1-s_c)}{s_c}}-4\int_{\R^n}\abs{x} \partial_rA_0 \abs{u}^2\ dx\norm{u}_2^{2\frac{(1-s_c)}{s_c}}\\
&\qquad\qquad+8\int_{\R^n}\abs{x}\mathcal Im(uB_\tau\cdot \overline{\nabla_A u})  \ dx\norm{u}_2^{2\frac{(1-s_c)}{s_c}}\\
&\leq 4n(p-1)E_A[u]\norm{u}_2^{2\frac{(1-s_c)}{s_c}}-2(n(p-1)-4)\norm{\nabla_Au}^2_2\norm{u}_2^{2\frac{(1-s_c)}{s_c}}\\
&\qquad +2n(p-1)\int A_{0,-}\abs{u}^2dx\norm{u}_2^{2\frac{(1-s_c)}{s_c}}+4\int_{\R^n}\abs{x} (\partial_rA_0)_- \abs{u}^2\ dx\norm{u}_2^{2\frac{(1-s_c)}{s_c}}\\
&\qquad\qquad+8\int_{\R^n}\abs{x}\mathcal Im(uB_\tau\cdot \overline{\nabla_A u})  \ dx\norm{u}_2^{2\frac{(1-s_c)}{s_c}}.
\end{split}
\ee
Therefore, by \eqref{a0}, \eqref{a1}, \eqref{b0}
 \be\nonumber
\begin{split}
\norm{u}_2^{2\frac{(1-s_c)}{s_c}}\, \partial_t^2 \int |x|^2 |u|^2 dx &\leq  4n(p-1)E_A[u]\norm{u}_2^{2\frac{(1-s_c)}{s_c}}-2(n(p-1)-4)\norm{\nabla_Au}^2_2\norm{u}_2^{2\frac{(1-s_c)}{s_c}}\\
&\quad+(2n(p-1)a_0+4a_1+8b_0)\norm{\nabla_Au}^2_2\norm{u}_2^{2\frac{(1-s_c)}{s_c}},
\end{split}
\ee
and by \eqref{c1}, when collected, all the coefficients of $\norm{\nabla_Au}^2$ above add up to be negative, and we can use that by \eqref{nablamore2}, $\norm{\nabla_Au}^2_2 \norm{u}_2^{2\frac{(1-s_c)}{s_c}}>(1-a_0)^\frac{2}{s_c(p-1)}\norm{\nabla u_Q}^{2}_2\norm{u_Q}^{\frac{2(1-s_c)}{s_c}}_2$.  This gives

 \be\label{uselater}
\begin{split}
\norm{u}_2^{2\frac{(1-s_c)}{s_c}}\, \partial_t^2 \int |x|^2 |u|^2 dx &\leq  4n(p-1)E_A[u]\norm{u}_2^{2\frac{(1-s_c)}{s_c}}-{\bigg(}2(n(p-1)-4)(1-a_0)^\frac{2}{s_c(p-1)}\norm{\nabla u_Q}^{2}_2\norm{u_Q}^{2\frac{(1-s_c)}{s_c}}_2\\
&\quad-(2n(p-1)a_0+4a_1+8b_0)(1-a_0)^\frac{2}{s_c(p-1)}\norm{\nabla u_Q}^{2}_2\norm{u_Q}^{2\frac{(1-s_c)}{s_c}}_2{\bigg)}.
\end{split}
\ee
Next, by \eqref{eless2}, we have there exists $\varepsilon>0$ such that
\be\nonumber
\begin{split}
\norm{u}_2^{2\frac{(1-s_c)}{s_c}}\, \partial_t^2 \int |x|^2 |u|^2 dx
%&\leq  4n(p-1)E_A[u]-2(n(p-1)-4)\norm{\nabla_Au}^2+(2n(p-1)a_0+4a_1+8b_0)\norm{\nabla_Au}^2\\
&\leq-\varepsilon+ 4n(p-1) (1-a_0)^{1+\frac 2{s_c(p-1)}}E[u_Q] \norm{u_Q}_2^{2\frac{(1-s_c)}{s_c}}-4(2a_0+a_1+2b_0)x_\ast\\
&\quad-2(n(p-1)-4)(1-a_0)^\frac{2}{s_c(p-1)}\norm{\nabla u_Q}^{2}_2\norm{u_Q}^{2\frac{(1-s_c)}{s_c}}_2\\
&\qquad+(2n(p-1)a_0+4a_1+8b_0)(1-a_0)^\frac{2}{s_c(p-1)}\norm{\nabla u_Q}^{2}_2\norm{u_Q}^{2\frac{(1-s_c)}{s_c}}_2.
\end{split}
\ee
Then \eqref{equal} allows us to make some cancellations and obtain 
  \begin{align*}
\norm{u}_2^{2\frac{(1-s_c)}{s_c}}\, \partial_t^2 \int |x|^2 |u|^2 dx
&<-\varepsilon-2a_0(n(p-1)-4)(1-a_0)^\frac{2}{s_c(p-1)}\norm{\nabla u_Q}^{2}_2\norm{u_Q}^{2\frac{(1-s_c)}{s_c}}_2-4(2a_0+a_1+2b_0)x_\ast\\
&\qquad+(2n(p-1)a_0+4a_1+8b_0)(1-a_0)^\frac{2}{s_c(p-1)}\norm{\nabla u_Q}^{2}_2\norm{u_Q}^{2\frac{(1-s_c)}{s_c}}_2\\
&= -\varepsilon,
\end{align*}
as needed.

\appendix 
\section{On Potentials in 3 Dimensions}\label{app1}
\par The results in this paper hold for potentials providing sufficient local theory and smallness condition on $\norm{|x|^2 B_{\tau}}_{\infty}$.

We have interest in the potentials considered by D'Ancona, Fanelli, Vega, and Visciglia \cite{JFA}, for which they have developed Strichartz estimates. In dimensions 4 and greater, they consider potentials with a smallness condition on $\norm{|x|^2 B_{\tau}}_{\infty}$, similar to our assumption. In dimension 3, this is replaced with a smallness condition on $\norm{|x|^{3/2} B_{\tau}}_{L^2_r L^{\infty}(S_r)}$, where the $L^2_r L^{\infty}(S_r)$ norm is given by \eqref{Srnorm}. Due to the differing assumptions, we discuss construction of potentials in dimension 3, which imply existence of potentials having nonzero trapping component, for which Strichartz estimates hold, and for which our results hold. We also note that these constructions have several straightforward generalizations.

Consider a function $f:\mathbb{R}^+ \rightarrow \mathbb{R}$ such that $f(|x|^2) \in C^1_{loc}(\mathbb{R}^N \setminus \{0\})$. We construct a potential
\[A \coloneqq f(|x|^2) \left( -x_2 , x_1, 0 \right).\]
For $x \neq 0$, we have
\[\grad \cdot A = 0.\]
If we also assume $|x|\, f(|x|^2) \in L^4_{loc}(\mathbb{R}^N)$, then the Leinfelder-Simander Theorem implies that $\laplacian_A$ is essentially self adjoint on $C^{\infty}_0(\mathbb{R}^N)$.
The potential $A$ has trapping component
\[ B_{\tau} = \frac{2 \left(f(|x|^2) + |x|^2 f'(|x|^2) \right)}{|x|} \left(-x_2 ,x_1,0 \right).\]
In particular, for $\alpha \in \mathbb{R}$, and $f(r) = \frac{1}{r^\alpha}$, we have
\[B_{\tau} = \frac{2(1-\alpha)}{|x|^{2\alpha + 1}} (-x_2,x_1,0).\]
By gluing, given any $\beta \neq 1$, and $\gamma \neq 1$, we can construct a potential $A$ such that $\beta$ is minimal satisfying
\[\norm{|x|^{\beta} A}_{L^{\infty}(B(1))}<\infty,\]
$\beta+1$ is minimal satisfying
\[\norm{|x|^{\beta+1} B_{\tau}}_{L^{\infty}(B(1))}<\infty,\]
$\gamma$ is minimal satisfying
\[\norm{|x|^{\gamma} A}_{L^{\infty}(\mathbb{R}^N \setminus B(1))}<\infty\,,\]
and $\gamma+1$ is minimal satisfying
\[\norm{|x|^{\gamma+1} B_{\tau}}_{L^{\infty}(\mathbb{R}^N \setminus B(1))}<\infty\,.\]
It is straightforward that there exist similar constructions respecting the norm $\norm{|x|^{\alpha} B_{\tau}}_{L^2_r L^{\infty}(S_r)}$ rather than $\norm{|x|^{\alpha} B_{\tau}}_{L^{\infty}(B(1))}$.
In other words, we can construct potentials with relative freedom over the strength of singularity at $0$ as well as decay rate, and for $N \ge 3$ the main theorem holds for a nontrivial class of magnetic potentials.

\end{document}